\documentclass[11pt]{amsart}
\usepackage[margin=1in]{geometry}
\usepackage{amsmath,amsthm,amsfonts,amssymb,amscd,nccmath,mathtools}
\usepackage{amsthm}
\usepackage{enumerate}
\usepackage{mathrsfs}
\usepackage{graphicx}
\usepackage{hyperref}
\usepackage{hhline}
\usepackage{graphbox}
\usepackage{caption}
\usepackage{subcaption}
\usepackage[export]{adjustbox}
\usepackage{numprint}
\usepackage{xcolor}
\usepackage[doi=false,isbn=false,url=false,eprint=false,style=trad-plain,giveninits=true]{biblatex}
\makeatletter
\let\c@table\c@figure 
\let\ftype@table\ftype@figure 
\makeatother

\mathtoolsset{showonlyrefs=true}
\numberwithin{equation}{section}
\numberwithin{figure}{section}
\numberwithin{table}{section}

\newtheorem{theorem}{Theorem}[section]

\newtheorem{proposition}[theorem]{Proposition}
\newtheorem{corollary}[theorem]{Corollary}
\newtheorem{lemma}[theorem]{Lemma}

\theoremstyle{definition}
\newtheorem{definition}[theorem]{Definition}
\newtheorem*{remark}{Remark}

\theoremstyle{plain}
\newcommand{\thistheoremname}{}
\newtheorem*{genericthm*}{\thistheoremname}
\newenvironment{namedthm*}[1]
  {\renewcommand{\thistheoremname}{#1}%
   \begin{genericthm*}}
  {\end{genericthm*}}

\hypersetup{%
  colorlinks=true,
  linkcolor=blue,
  citecolor=blue,
  urlcolor=blue,
  linkbordercolor={0 0 1}
}

\newcommand{\obar}[1]{\overline{#1}}
\newcommand{\N}{\mathbb{N}}
\newcommand{\R}{\mathbb{R}}
\newcommand{\Z}{\mathbb{Z}}

\newcommand{\cC}{\mathcal{C}}
\newcommand{\cP}{\mathcal{P}}
\newcommand{\cS}{\mathcal{S}}

\newcommand{\abs}[1]{\left\lvert #1 \right\rvert}

\newcommand{\lrp}[1]{\left(#1\right)}
\newcommand{\lrb}[1]{\left[#1\right]}

\newcommand{\qqand}{\qquad\text{and}\qquad}

\title{Arithmetic Polygons and Sums of Consecutive Squares}
\author{Jack Anderson}
\author{Amy Woodall}
\author{Alexandru Zaharescu}

\address[Jack Anderson]{Department of Mathematics, University of Illinois at Urbana-Champaign, Urbana, IL 61801, USA.}
\email{jacka4@illinois.edu}

\address[Amy Woodall]{Department of Mathematics, University of Illinois at Urbana-Champaign, Urbana, IL 61801, USA.}
\email{amyew3@illinois.edu}

\address[Alexandru Zaharescu]{Department of Mathematics,University of Illinois at Urbana-Champaign, Urbana, IL 61801, USA,
and 
``Simion Stoilow'' Institute of Mathematics of the Romanian Academy,~21 
Calea Grivitei 
Street, P. O. Box 1-764, Bucharest 014700, Romania.}
\email{zaharesc@illinois.edu}



\begin{document}

\begin{abstract}
    We introduce and study arithmetic polygons. We show that these arithmetic polygons are connected to triples of square pyramidal numbers. For every odd $N\geq3$, we prove that there is at least one arithmetic polygon with $N$ sides. We also show that there are infinitely many arithmetic polygons with an even number of sides.
\end{abstract}

\maketitle

\section{Introduction}\label{Introduction}

The Pythagorean Theorem states that the three side lengths $x,y,z$ of a right triangle are related by $x^2 + y^2 = z^2$. One can find an infinite family of integer solutions to this equation that are parametrically related; this classical construction is Euclid's formula. Interestingly, we can also find an example where $x$, $y$, and $z$ are consecutive integers, namely \[
    3^2 + 4^2 = 5^2.
\]

We consider the extension of this example to polygons with more than $3$ sides by defining an arithmetic polygon.
\begin{definition}\label{D:ArithmeticPolygons}
    An \textit{arithmetic polygon} is a polygon with integer side lengths and a special vertex denoted $O$ that satisfies the following properties:
    \begin{enumerate}
        \item\label{I:ArithmeticPropertyLengths} As one traverses the polygon starting and finishing at $O$, each side has length one greater than the length of the side preceding it.
        \item\label{I:ArithmeticPropertyIntersection} For each side of the polygon, there is a line perpendicular to the side which passes through both $O$ and one of the vertices at either end of that side.
        \item\label{I:ArithmeticPropertyDegenerate} There are no degenerate vertices. That is, there are no vertices with an angle of $0$ or $\pi$ radians.
    \end{enumerate}
\end{definition}

One sees that the famous $3$-$4$-$5$ triangle is an example of an arithmetic polygon. We will see that if the sides of a given arithmetic polygon have lengths $a+1$ through $c$, then there is some $b$ satisfying $a+1<b<c$ such that
\begin{equation}\label{eq:original-statement-polygons}
    (a+1)^2 + (a+2)^2 + \cdots + b^2 = (b+1)^2 + (b+2)^2 + \cdots + c^2.
\end{equation}
In the case of the $3$-$4$-$5$ triangle, we have $a=2$, $b=4$, $c=5$. In other words, from any arithmetic polygon, one can find a solution to \eqref{eq:original-statement-polygons}, which we call the Sum of Consecutive Squares (SoCS) Problem. Furthermore, we shall show that given any solution to the SoCS Problem, one can construct at least one---but potentially many---arithmetic polygons.

\begin{theorem}\label{T:SolutionsPolygonsCorrespondence}
    There is a surjective, many-to-one correspondence from arithmetic polygons to solutions to \eqref{eq:original-statement-polygons}.
\end{theorem}

In order to find arithmetic polygons, we study how one can find solutions to the SoCS Problem. We can simplify our search by using square pyramidal numbers. The $n$th square pyramidal number $P_n$ is the number of spheres that it takes to create a pyramid with a square base of side length $n$. The first few square pyramidal numbers are \[
    1, \, 5, \, 14, \, 30, \, 55, \, 91, \, 140, \, 204, \, 285, \, 385, \dots.
\]
Square pyramidal numbers have been known since antiquity, and continue to be studied through to the present. More details about this sequence can be found on the On-Line Encyclopedia of Integer Sequences \cite[Entry A000330]{oeis}.
From the definition, we see that \[
    P_n = \sum_{k = 1}^n k^2,
\]
and by induction one can evaluate the sum exactly as \[
    P_n = \frac{n(n+1)(2n+1)}{6}.
\]
In terms of $P_n$, we can restate Equation \eqref{eq:original-statement-polygons} as $P_b - P_{a} = P_c - P_b$. Simplifying this equation, we wish to find $a,b,c \in \Z$ that satisfy 
\begin{equation}\label{E:PyramidalEquality}
    P_a + P_c = 2 P_b
\end{equation}
with $0 < a+1 < b < c$.

\begin{table}[t]
    \centering
    \begin{tabular}{cccc}
       (2, 4, 5) & (54, 60, 65) & (170, 180, 189) & (17, 34, 42) \\
       (9, 12, 14) & (77, 84, 90) & (209, 220, 230) & (350, 364, 377) \\
       (20, 24, 27) & (104, 112, 119) & (252, 264, 275) & (405, 420, 434) \\
       (35, 40, 44) & (135, 144, 152) & (299, 312, 324) & (464, 480, 495)
    \end{tabular}
    \caption{Some solutions to the Sum of Consecutive Squares Problem.}
    \label{tab:some-solutions}
\end{table}

Once restated as an equality of square pyramidal numbers, solutions to the SoCS problem can be easily computed for small $a,b,c$. See Table \ref{tab:some-solutions} for a list of all solutions $(a,b,c)$ with $c - a \leq 33$ (Theorem \ref{T:BoundOnSolutionsForFixedN} shows that this list is complete).
Many of the solutions given in Table \ref{tab:some-solutions} appear to be related. In fact, all solutions except for the solution $(17, 34, 42)$ are of the form \[
    (2k^2 + k - 1, 2k^2 + 2k, 2k^2 + 3k)
\]
for some $k \in \N$. In general, every $k \in \N$ gives rise to a solution to the SoCS problem of this form. We call these solutions the \textit{parameterized solutions}. However, the solution $(17, 34, 42)$ shows that not every solution to the SoCS problem is a parameterized solution. We must therefore develop more sophisticated methods to find more solutions outside of these parameterized solutions.

In Section~\ref{S:SoCS}, we describe two algorithms for finding new solutions:
\begin{enumerate}
    \item \textit{Finding new solutions of fixed length.} If we fix $c-a=N$, and $b-a=\ell$, then Equation~\eqref{E:PyramidalEquality} reduces to a quadratic in $b$ for which we can easily find solutions. Iterating this process over $\ell$ between $1$ and $N-1$ gives us all solutions for $c-a=N$.
    \item \textit{Generating new solutions from known solutions.} Given a known ``base solution'' $(a_0, b_0, c_0)$, a point in $\R^3$, we consider the plane passing through this point and the line $x=y=z$, which contains the trivial solutions to Equation~\eqref{E:PyramidalEquality}. The intersection of this plane and the surface given by Equation~\eqref{E:PyramidalEquality} can be described with a quadratic equation involving two variables. After a change of variables, this equation reduces to a generalized Pell equation, for which we can find an infinite family of solutions. Reversing this change of variables gives us new solutions to the SoCS Problem.
\end{enumerate}

Using the second algorithm, we show the following results.
\begin{theorem}\label{T:ExistenceOfSolutions}
    For each odd $N \geq 3$, there is at least one solution $(a,b,c)$ to \eqref{eq:original-statement-polygons} with $N = c - a$. Furthermore, for each odd $N \geq 3$, there is an infinite family $(a_n, b_n, c_n)$ of solutions to \eqref{eq:original-statement-polygons} with $0 < a_n + 1 < b_n < c_n$ and with $c_n - a_n$ odd, all of which lie in the same plane (which is unique for each $N$).
\end{theorem}

\begin{theorem}\label{T:EvenSolutions}
    There is an infinite number of solutions $(a,b,c)$ to \eqref{eq:original-statement-polygons} with $0 < a + 1 < b < c$ and with $c-a$ even.
\end{theorem}

\begin{remark}
    In light of Theorem~\ref{T:SolutionsPolygonsCorrespondence}, Theorems~\ref{T:ExistenceOfSolutions} and \ref{T:EvenSolutions} can be restated in terms of arithmetic polygons. If $(a,b,c)$ is a solution to the SoCS Problem, then all corresponding arithmetic polygons have $N = c - a$ sides. Then, for each odd $N \geq 3$ we have an arithmetic polygon with $N$ sides, and each $N$ generates an infinite family of arithmetic polygons with an odd number of sides. Furthermore, there is an infinite number of arithmetic polygons with an even number of sides.
\end{remark}

In Section~\ref{S:ArithmeticPolygons}, we study the relation between solutions to the SoCS Problem and arithmetic polygons, beginning with a proof of Theorem~\ref{T:SolutionsPolygonsCorrespondence}. 
We then discuss the convexity (or, rather, the lack thereof) of arithmetic polygons. In particular, we show the following.

\begin{theorem}\label{T:Only2ConvexPolygons}
    There are only 2 distinct (up to rigid transformations) convex arithmetic polygons.
\end{theorem}

The SoCS Problem is similar to the cannonball problem, which asks if there are any square pyramidal numbers that are themselves a square. In other words, do there exist positive integers $b$ and $c$ such that
\begin{equation}
    1^2 + 2^2 + \cdots + (b-1)^2 + b^2 = c^2.
\end{equation}

In 1918, Watson \cite{watson-proof} proved the assertion (first proposed by Lucas \cite{lucas-proposal} in 1875) that the only nontrivial solution is $(b, c) = (24, 70)$, where we see that
\begin{equation}
    1^2 + 2^2 + \cdots + 23^2 + 24^2 = 4900 = 70^2.
\end{equation}
(A more elementary proof was given by Anglin in 1990, which can be found in \cite{anglin-elementary-proof}.)

One could then define a related cannonball polygon by amending the first condition in our definition of an arithmetic polygon:

\begin{enumerate}
    \item As one traverses around the polygon, starting at point $O$, the first side has length 1. Then, every side has length exactly one greater than the side preceding it except for the final side which may have any integer length.
\end{enumerate}

From any solution to the cannonball problem, we can construct a number of related cannonball polygons (see Figure~\ref{F:CannonballPolygonGood}) using methods similar to Theorem \ref{T:SolutionsPolygonsCorrespondence}. However, there are cannonball polygons which do not relate to any solution of the cannonball problem. For an example of such a polygon, see Figure~\ref{F:CannonballPolygonalBad}, which has side lengths $1$ through $36$, and then a final side has length $52$. In this example, $52^2 = 2704$ whereas $1^2+\cdots+36^2 = 16206$. Therefore, we do not have the same correspondence with cannonball polygons as we have with arithmetic polygons. We do briefly note that in the example of Figure~\ref{F:CannonballPolygonalBad}, one finds
\begin{equation}
    1^2 + \cdots + 30^2 = 31^2 + \cdots + 36^2 + 52^2.
\end{equation}

Other authors have considered generalizations of the cannonball problem. In \cite{bennett, oranges}, the authors used the theory of Diophantine approximation to prove their results; we will apply similar methods to prove Theorems \ref{T:ExistenceOfSolutions} and \ref{T:EvenSolutions}.

\begin{figure}[t]
    \centering
    \begin{subfigure}{0.44\linewidth}
        \centering
        \includegraphics[width=\textwidth]{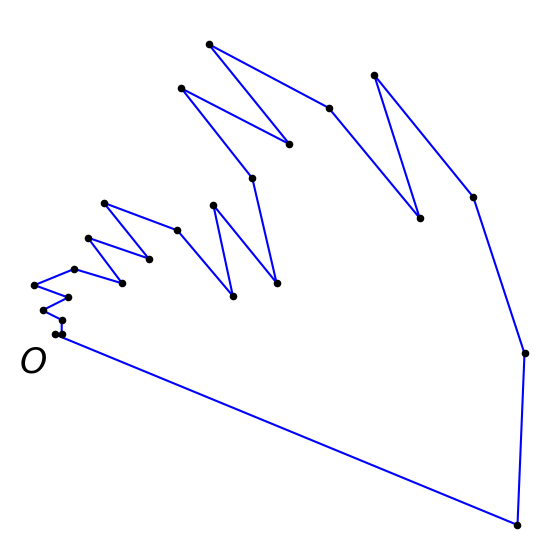}
        \caption{A cannonball polygon constructed from the (nontrivial) solution to the cannonball problem.}\label{F:CannonballPolygonGood}
    \end{subfigure}
    \hfill
    \begin{subfigure}{0.44\linewidth}
        \centering
        \includegraphics[width=\textwidth]{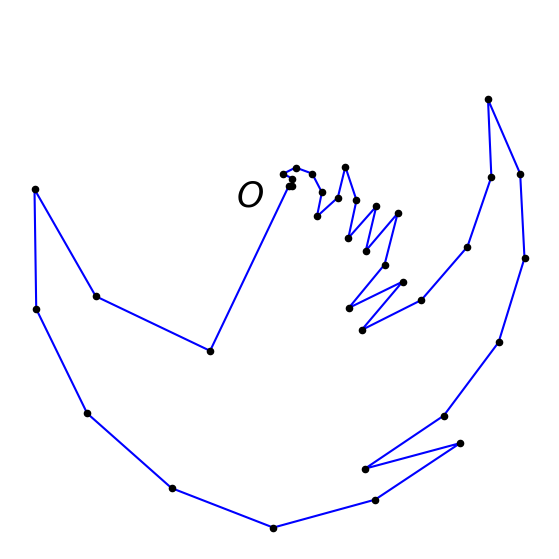}
        \caption{A cannonball polygon with side lengths 1 though 36 and with a final side length of 52. This polygon does not correspond to any solution to the cannonball problem.}\label{F:CannonballPolygonalBad}
    \end{subfigure}
    \caption{Two examples of cannonball polygons. In both examples, the side lengths increase as one travels clockwise around the polygon starting from $O$.}\label{F:CannonballPolygons}
\end{figure}

\section{Sums of Consecutive Squares}\label{S:SoCS}

We first describe the algorithm for finding every solution $(a,b,c)$ with a fixed length $c-a=N$. In particular, we prove the following bound.

\begin{theorem}\label{T:BoundOnSolutionsForFixedN}
    For any integer $N>1$, there can be at most $2(N-1)$ possible solutions $(a,b,c)$ to \eqref{E:PyramidalEquality} with $0 < a+1 < b < c$ such that $c-a=N$.
\end{theorem}

\begin{proof}
    We first write $N$ as a sum of two integers via $N=(c-b)+(b-a)$. Fix an integer $\ell$ such that $0<\ell< N$. We write
    \begin{align*}
        c-b &= N-\ell, \\
        b-a &= \ell.
    \end{align*}
    By finding all solutions for each $\ell$, we can find all the solutions for our given $N$. We now substitute
    \begin{equation}\label{E:bSubNPolygonProof}
    \begin{aligned}
        a &= b-\ell, \\
        c &= b+N-\ell
    \end{aligned}
    \end{equation}
    into Equation \eqref{E:PyramidalEquality}. 
    The $b^3$ terms cancel, and we are left with a quadratic in $b$ with the coefficients being in terms of $N$ and $\ell$. In particular, we have
    \begin{align*}
        0 &= 6(N-2\ell)b^2 + 6(N^2+2\ell^2-2N\ell+N-2\ell)b \\
        &\qquad+ 2N^3-4\ell^3- 6N^2\ell+6N\ell^2+3N^2+6\ell^2-6N\ell+N-2\ell.
    \end{align*}
    Solving this quadratic in $b$, we find
    \begin{equation}
        b = \frac{-3N^2-6\ell^2+6N\ell-3N+6\ell \pm \sqrt{3}\sqrt{\Delta}}{6(N-2\ell)},
    \end{equation}
    with \[
        \Delta = -N^4 + 8N^3\ell - 12N^2\ell^2 + 8N\ell^3 - 4\ell^4 + N^2 - 4N\ell + 4\ell^2.
    \]
    Equation \eqref{E:bSubNPolygonProof} gives us the corresponding $a$ and $c$ values.
    Note that it is not guaranteed that $b$ is a positive integer or that $a>0$, but we have a necessary condition for a solution. There are at most two possible solutions for each $\ell$. Furthermore, there are $N-1$ possible values that $\ell$ can take. Therefore, there are at most $2(N-1)$ possible solutions for our given $N$.
\end{proof}

Using the algorithm above, we can calculate all of the solutions $(a,b,c)$ with $c - a \leq X$ for any given $X \in \N$. Table \ref{tab:many-non-param-solutions} gives all of the non-parameterized solutions with $c - a \leq \numprint{10048}$; here we omit the many parameterized solutions within this range since these solutions are easy to reproduce.

\begin{table}[t]
    \centering
    \begin{tabular}{ccc}
        (17, 34, 42) & (2844, 3839, 4484) & (384, 5222, 6579) \\
        (3, 38, 48) & (677, 2250, 2822) & (2555, 7827, 9804) \\
        (11, 50, 63) & (871, 2610, 3268) & (1821, 7489, 9413) \\
        (59, 110, 135) & (2159, 3892, 4760) & (\numprint{22787}, \numprint{27649}, \numprint{31224}) \\
        (66, 159, 198) & (3699, 5384, 6395) & (\numprint{16394}, \numprint{21575}, \numprint{25029}) \\
        (15, 142, 179) & (965, 3030, 3797) & (\numprint{116547}, \numprint{121124}, \numprint{125379}) \\
        (473, 855, 1046) & (2050, 4290, 5305) & (2930, 9487, \numprint{11894}) \\
        (1634, 2470, 2954) & (2295, 5729, 7140) & (\numprint{35948}, \numprint{41579}, \numprint{45996})
    \end{tabular}
    \caption{All non-parametric solutions to the SoCS Problem with $c - a \leq \numprint{10048}$.}
    \label{tab:many-non-param-solutions}
\end{table}

We now focus on the second algorithm, the ``Pell Generator'' algorithm. This algorithm allows us to compute infinitely many solutions from any given solution. To calculate new solutions, we will solve a generalized Pell equation \begin{equation}
    u^2 - Av^2 = B,
\end{equation}
where $A$ and $B$ are given in terms of our initial solution.
In order for this Pell equation to have infinitely many solutions, we must ensure that $A$ is positive and not a square. These claims can be proven with the help of the following lemma, which states that solutions $(a,b,c)$ cannot be too ``unbalanced.'' That is, $b-a$ cannot be too large in comparison to $c-b$.
\begin{lemma}\label{L:SolutionsBalanced}
    Let $(a,b,c)$ be a solution to \eqref{E:PyramidalEquality} with $0 < a + 1 < b < c$. Then,
    \begin{equation}
        1 < \frac{b-a}{c-b} < 1 + 2^{1/3} + 2^{2/3} \approx 3.8473.
    \end{equation}
\end{lemma}

\begin{proof}
    Expanding \eqref{E:PyramidalEquality}, we see that \begin{equation}
        a(a + 1)(2a + 1) + c(c + 1)(2c + 1) = 2b(b + 1)(2b + 1).
    \end{equation}
    Rearranging, we have \begin{equation}
        (1 + 3a + 2a^2 + 3b + 2b^2 + 2ab)(b - a) = (1 + 3b + 2b^2 + 3c + 2c^2 + 2bc)(c - b).
    \end{equation}
    Let $\ell = b - a$ and $m = c - b$. Then \begin{equation}
        \frac{\ell}{m} = \frac{f(b,c)}{f(a,b)},
    \end{equation}
    where \begin{equation}
        f(x, y) = 1 + 3x + 2x^2 + 3y + 2y^2 + 2xy.
    \end{equation}
    If $x$ is fixed, then $f(x,y)$ is increasing as a function of $y$. Similarly, $f(x,y)$ is increasing in $x$ for fixed $y$. Recalling that $0 < a + 1 < b < c$, we have the lower bound \begin{equation}
        \frac{\ell}{m} = \frac{f(b,c)}{f(a,b)} > \frac{f(b, b)}{f(b, b)} = 1.
    \end{equation}

    To obtain an upper bound for $\ell/m$, we require an upper bound on $c$. Since $\ell > m$, we have the upper bound $c < 2b - a$. However, this bound is not sufficient for our needs. Instead, we note that \begin{equation}
        2c^3 < c(c + 1)(2c + 1) = 2b(b + 1)(2b + 1) - a(a + 1)(2a + 1) < 2b(b + 1)(2b + 1).
    \end{equation}
    Then, we bound $\ell/m$ as \begin{equation}
        \frac{\ell}{m} = \frac{f(b,c)}{f(a,b)} < \frac{f\lrp{b,\sqrt[3]{b(b + 1)(2b + 1)}}}{f(0, b)},
    \end{equation}
    where we are using that $a \geq 0$. The expression on the right-hand side is increasing for all $b > 0$, and the limit as $b \to \infty$ is $1 + 2^{1/3} + 2^{2/3}$. Thus, we have $\ell/m < 1 + 2^{1/3} + 2^{2/3}$.
\end{proof}

\begin{corollary}\label{C:APos}
    Let $(a,b,c)$ be a solution to \eqref{E:PyramidalEquality} with $0 < a + 1 < b < c$. Let \begin{align}
        A &= 3 (b - c)^2 \lrp{12 (b - a)^2 (c - b)^2 - (a-2b+c)^4}.
    \end{align}
    Then $A > 0$ and $A$ is not a square.
\end{corollary}

\begin{proof}
    Let $\ell = b - a$ and $m = c - b$.
    We can restate $A$ in terms of $\ell$ and $m$ as \begin{equation}\label{eq:A-def-l-m}
        A = 3 m^2 \lrp{12\ell^2 m^2 - (m - \ell)^4}.
    \end{equation}
    We note that $A > 0$ if and only if \begin{equation}
        12\ell^2 m^2 - (m - \ell)^4 > 0.
    \end{equation}
    Dividing through by $m^4$ and setting $s = \ell/m$, we study the quartic \begin{equation}
        g(s) = -s^4 + 4s^3 + 6s^2 + 4s - 1.
    \end{equation}
    For $s \neq 0$, we may rewrite $g(s)$ as \begin{equation}
        g(s) = s^2 \lrp{-\lrp{s + \frac{1}{s}}^2 + 4\lrp{s + \frac{1}{s}} + 8}.
    \end{equation}
    The equation $g(s) = 0$ can then be solved by two applications of the quadratic formula. The real roots are \begin{equation}
        s = 1 + \sqrt{3} - \sqrt{3 + 2\sqrt{3}} \approx 0.1896 \qqand s = 1 + \sqrt{3} + \sqrt{3 + 2\sqrt{3}} \approx 5.2745.
    \end{equation}
    For $s$ between these real roots, we have $g(s) > 0$. Since $s = \ell/m$ satisfies $1 < s < 1 + 2^{1/3} + 2^{2/3}$ by Lemma \ref{L:SolutionsBalanced}, we see that $s$ always lies within the range where $g(s) > 0$. Thus, $A > 0$.

    We now show that $A$ is not a square. We divide \eqref{eq:A-def-l-m} by $m^6$ 
    and set $s = \ell/m$. Let $t = \sqrt{A}/m^3$. Then if $A$ is a square, the quartic curve \begin{equation}\label{eq:quartic-curve}
        t^2 = -3s^4 + 12s^3 + 18s^2 + 12s - 3
    \end{equation}
    has a rational point $(s,t)$. We show that the only rational points on \eqref{eq:quartic-curve} are $(1, \pm 6)$. Since $s = 1$ corresponds to $\ell = m$, which we know does not occur by Lemma \ref{L:SolutionsBalanced}, we will conclude that $A$ is not a square.

    To find the rational points on \eqref{eq:quartic-curve}, we make a change of variables to turn the equation into an elliptic curve. Let \begin{equation}\label{eq:s-t-change-of-var}
        s = \frac{-3x - y + 9z}{3x - y - 9z}
        \qqand
        t = \frac{6(2x^3 - 9x^2 z - y^2 z + 27z^3)}{z(3x - y - 9z)^2}.
    \end{equation}
    The reverse change of variables is given by \begin{equation}\label{eq:x-y-z-change-of-var}
        \frac{x}{z} = \frac{3(s^2 + 4s + t + 1)}{(s - 1)^2}
        \qqand
        \frac{y}{z} = \frac{9(6s^2 + st + 6s + t)}{(s - 1)^3}.
    \end{equation}
    Here we consider $(s,t)$ as an affine point and $[x:y:z]$ as a projective point. This change of variables can be found by combining Proposition 1.2.1 of \cite{elliptic-curve-handbook} and the standard manipulations to put an elliptic curve into Weierstrass form. Under this change of variables, \eqref{eq:quartic-curve} becomes \begin{equation}\label{eq:elliptic-curve}
        \frac{144(x - 3z)^3}{z^2(3x - y - 9z)^4} (x^3 - 27z^3 - y^2 z) = 0.
    \end{equation}
    Via the L-functions and modular forms database \cite[\href{https://www.lmfdb.org/EllipticCurve/Q/36/a/3}{Elliptic Curve 36.a3}]{lmfdb}, we see that the elliptic curve $y^2 z = x^3 - 27z^3$ has precisely two rational points: $[3:0:1]$ and the point at infinity $[0:1:0]$.
    If $(s,t)$ is a rational point with $s \neq 1$ that lies on the quartic curve \eqref{eq:quartic-curve}, then \eqref{eq:x-y-z-change-of-var} gives a rational point $[x : y : z]$ that satisfies \eqref{eq:elliptic-curve}. Since $[3 : 0 : 1]$ and $[0 : 1 : 0]$ both satisfy $x - 3z = 0$, the rational point must necessarily have $x = 3z$. However, solving \[
        \frac{x}{z} = 3 = \frac{3(s^2 + 4s + t + 1)}{(s - 1)^2}
    \]
    with $s \neq 1$ yields $t = -6s$, which reduces \eqref{eq:quartic-curve} to the equation $0 = -3(s - 1)^4$. Thus there are no rational points on \eqref{eq:quartic-curve} satisfying $t = -6s$ and $s \neq 1$, so there cannot be a rational point $(s,t)$ with $s \neq 1$ lying on the quartic curve. We conclude that the only rational points on \eqref{eq:quartic-curve} are $(s,t) = (1, \pm 6)$, so $A$ is not a square.
\end{proof}

We now describe the ``Pell Generator'' algorithm, which allows one to compute infinitely many solutions from any given computed solution. We then apply this algorithm to prove Theorems~\ref{T:ExistenceOfSolutions} and \ref{T:EvenSolutions}.

\begin{proposition}\label{prop:Pell-generator}
    Let $(a,b,c) = (a_0, b_0, c_0)$ be a solution to \eqref{E:PyramidalEquality} with $0 < a + 1 < b < c$.
    Then there exists an infinite family $(a_n, b_n, c_n)$ of integer solutions to \eqref{E:PyramidalEquality},
    all of which lie in the plane \begin{equation}\label{eq:a-b-c-plane}
        (b-c)x + (c-a)y + (a-b)z = 0.
    \end{equation}
    Furthermore, $c_n - a_n \equiv c - a \pmod{2}$ for all $n$.
\end{proposition}

\begin{proof}
    We first note that $(a, b, c)$ lies on the plane \eqref{eq:a-b-c-plane}, and this plane is precisely the plane that passes through the points $(a, b, c)$ and $(1,1,1)$. We study the intersection of this plane and the surface \begin{equation}\label{eq:beginning-surface}
        x(x+1)(2x+1) + z(z+1)(2z+1) = 2y(y+1)(2y+1),
    \end{equation}
    which is the same surface as $P_x + P_z = 2P_y$.
    
    Suppose that $(x,y,z)$ is an integer solution to \eqref{eq:a-b-c-plane} and \eqref{eq:beginning-surface}. We make the change of variables $\ell = b - a$ and $m = c - b$ and note that $\ell + m = c - a$.
    Rearranging, Equation \eqref{eq:a-b-c-plane} becomes \begin{equation}\label{eq:plane-before-division}
        m (y - x) = \ell (z - y).
    \end{equation}
    Likewise, we can rearrange \eqref{eq:beginning-surface} to obtain \begin{equation}\label{eq:expanded-surface}
        (1 + 3x + 2x^2 + 3y + 2y^2 + 2xy) (y-x) = (1 + 3y + 2y^2 + 3z + 2z^2 + 2yz) (z-y).
    \end{equation}
    Since $a+1 < b < c$, we have $m = c - b > 0$ and $\ell = b - a > 0$. Thus, taking $x = y$ or $y = z$ forces $x = y = z$ in Equation~\eqref{eq:plane-before-division}. However, while $x = y = z$ corresponds to solutions of \eqref{E:PyramidalEquality}, these values are not solutions of \eqref{eq:original-statement-polygons}. We will therefore assume that $x$, $y$, and $z$ are distinct. Dividing \eqref{eq:expanded-surface} by \eqref{eq:plane-before-division} yields \begin{equation}
        \ell (1 + 3x + 2x^2 + 3y + 2y^2 + 2xy) = m (1 + 3y + 2y^2 + 3z + 2z^2 + 2yz).
    \end{equation}
    Solving for $x$ in \eqref{eq:plane-before-division} yields \begin{equation}\label{eq:x-in-terms-of-y-z}
        x = \frac{(\ell + m) y - \ell z}{m}.
    \end{equation}
    We can remove the $x$ dependence to obtain  \begin{align}\label{eq:full-y-z-equation}
        0 &= \ell m^2 - m^3 + (3\ell^2 m + 6\ell  m^2 - 3 m^3)y + (2\ell^3 + 6\ell^2 m + 6 \ell m^2 - 2m^3) y^2 \\
        &\qquad+ (-3\ell^2 m - 3m^3) z + (2\ell^3 - 2 m^3) z^2 + (-4\ell^3 - 6\ell^2 m - 2 m^3) yz.
    \end{align}
    
    We note that $m^3 - \ell^3 \neq 0$ because the only real solution is when $m = \ell$, which we know does not occur by Lemma \ref{L:SolutionsBalanced}.
    Completing the square and multiplying through by $-8(m^3 - \ell^3) \neq 0$, we~obtain \begin{align*}
        0 &= -4 m^2 (m - \ell) (m^3 - \ell^3) - 3 m^2 \lrp{12\ell^2 m^2 - (m - \ell)^4} \lrp{2y + 1}^2 \\
        &\qquad+ \left(4 (m^3 - \ell^3) z + (4\ell^3 + 6\ell^2 m + 2m^3) y + (3\ell^2 m + 3m^3) \right)^2.
    \end{align*}
    Make the change of variables \begin{align*}
        u &= 4 (m^3 - \ell^3) z + (4\ell^3 + 6\ell^2 m + 2m^3) y + (3\ell^2 m + 3m^3), \\
        v &= 2y + 1, \\
        A &= 3 m^2 \lrp{12\ell^2 m^2 - (m - \ell)^4}, \\
        B &= 4 m^2 (m - \ell) (m^3 - \ell^3).
    \end{align*}
    We emphasize that the definition of $A$ given above is the same as in Corollary \ref{C:APos}. With this change of variables, we obtain the generalized Pell equation \begin{equation}\label{eq:pell-to-solve}
        u^2 - Av^2 = B.
    \end{equation}

    The explicit change of variables given above allows us to relate solutions $(x,y,z)$ to \eqref{eq:a-b-c-plane} and \eqref{eq:beginning-surface} to solutions $(u,v)$ to the Pell equation \eqref{eq:pell-to-solve}. We can use our initial solution $(a,b,c)$ to get an initial solution $(u_0, v_0)$ to the Pell equation. Then, we can apply the standard method of taking powers of the fundamental solution to $p^2 - Aq^2 = 1$ to calculate more solutions to the Pell equation. By reversing the change of variables, we will obtain new solutions to \eqref{eq:a-b-c-plane} and \eqref{eq:beginning-surface}.

    We take $(x,y,z) = (a,b,c)$ in our change of variables to obtain values for $u_0$ and $v_0$.
    Recalling that $\ell = b - a$ and $m = c - b$, the above change of variable simplifies to \begin{align}
        u_0 &= (c-b) \lrp{4a^3 + 3a^2 - 6a^2 b - 6ab + 4b^3 + 6b^2 - 6bc^2 - 6bc + 4c^3 + 3c^2}, \\
        v_0 &= 2b + 1.
    \end{align}
    For a given new solution $(u_n, v_n)$ to the Pell equation, we can obtain a solution $(a_n, b_n, c_n)$ by solving \begin{equation}\label{E:un-vn-a-b-c}
    \begin{aligned}
        u_n &= 4 \lrp{(c-b)^3 - (b-a)^3} c_n + \lrp{4(b-a)^3 + 6(b-a)^2 (c-b) + 2(c-b)^3} b_n \\
        &\qquad+ \lrp{3(b-a)^2 (c-b) + 3(c-b)^3} \\
        v_n &= 2b_n + 1
    \end{aligned}
    \end{equation}
    for $b_n$ and $c_n$ and then taking \[
        a_n = \frac{(c - a)b_n + (a - b)c_n}{c - b}
    \]
    as in \eqref{eq:x-in-terms-of-y-z}.

    To obtain new solutions to the Pell equation, we solve \eqref{eq:pell-to-solve} as a generalized Pell equation.
    Let $(p,q)$ be the fundamental solution to \[
        p^2 - Aq^2 = 1.
    \]
    By Corollary \ref{C:APos}, we know that $A$ is positive and not a square, so we are guaranteed that such a solution exists.
    Then for any $n \in \Z$, $(u_n, v_n)$ is a solution to \eqref{eq:pell-to-solve}, where $u_n, v_n \in \Z$ are defined by \begin{align}\label{E:un-vn-Pell}
        u_n + v_n\sqrt{A} = \lrp{u_0 + v_0\sqrt{A}} \lrp{p + q\sqrt{A}}^n.
    \end{align}
    
    For each $n \in \Z$, let $p_n, q_n \in \Z$ be such that $p_n + q_n\sqrt{A} = (p + q\sqrt{A})^n$.
    Equating \eqref{E:un-vn-a-b-c} and \eqref{E:un-vn-Pell} and using the formula for $a_n$ given above, we can solve for $(a_n, b_n, c_n)$ in terms of $p_n$, $q_n$, and $(a,b,c)$. We have \begin{align}
        a_n &= -\frac{1}{2} + \frac{1}{2} \lrp{1 + 2a} p_n + \frac{1}{2} \alpha (b - c) q_n, \\
        b_n &= -\frac{1}{2} + \frac{1}{2}(1 + 2b) p_n + \frac{1}{2} \beta (b-c) q_n, \\
        c_n &= -\frac{1}{2} + \frac{1}{2}(1 + 2c) p_n + \frac{1}{2} \gamma (b-c) q_n,
    \end{align}
    where \begin{align}
        \alpha &= 2 a^3+3 a^2-12 a b^2-12 a b+6 a c^2+6 a c+8 b^3+6 b^2-4 c^3-3 c^2, \\
        \beta &= -4 a^3+6 a^2 b-3 a^2+6 a b-4 b^3-6 b^2+6 b c^2+6 b c-4 c^3-3 c^2, \\
        \gamma &= -4 a^3+6 a^2 c-3 a^2+6 a c+8 b^3-12 b^2 c+6 b^2-12 b c+2 c^3+3 c^2.
    \end{align}
    
    The new solution $(a_n, b_n, c_n)$ is composed of, a priori, half-integers.
    However, the fact that $p_n^2 - Aq_n^2 = 1$ in fact guarantees that the new solution is always integral. We compute \begin{align}
        2a_n &\equiv -1 + p_n + (b + c)(a + c) q_n 
        \equiv -1 + p_n^2 - A q_n^2 \equiv 0 \pmod{2}.
    \end{align}
    Here we are using that \begin{equation}
        A = -3 (b - c)^2 \lrp{(a-2b+c)^4 - 12 (b - a)^2 (c - b)^2} \equiv (b + c) (a + c) \pmod{2}.
    \end{equation}
    Similar calculations show that $2b_n, 2c_n \equiv 0 \pmod{2}$. Since each of these values is even, we must have that $a_n$, $b_n$, and $c_n$ are integers. Thus, $(a_n, b_n, c_n)$ is an integer solution for every $n \in \Z$.

    Given that $a_n$, $b_n$, and $c_n$ are integral, we now determine the parity of the solution, meaning the parity of $c_n - a_n$. We compute \begin{align}\label{eq:cn-an}
        c_n - a_n &= (c-a) \lrp{p_n + 3(b-c) (a^2 + a - 2b^2 - 2b + c^2 + c) q_n} \\
        &\equiv (c-a) p_n \pmod{2},
    \end{align}
    where we note that $a^2 + a \equiv c^2 + c \equiv 0 \pmod{2}$.

    We claim that $p \equiv 1 \pmod{2}$ always. For contradiction, suppose that $p \equiv 0 \pmod{2}$.
    From $p^2-Aq^2=1$, we must have $2\nmid q$ and thus $A\equiv 3\pmod{4}$ which contradicts the fact that
    \begin{equation}
        A \equiv (b-c)^2 (a-2b+c)^4 \equiv 0 \text{ or } 1 \pmod{4}.
    \end{equation}
    
    We further claim that $p_n \equiv 1 \pmod{2}$ for all $n$. We proceed by induction. We see that $p_0 = p$ satisfies this claim, and we suppose that for some $n \geq 1$, we have $p_{n-1} \equiv 1 \pmod{2}$. Recurrently, $p_n$ is given by
    \begin{align}
        p_n &= p_{n-1} p + A q_{n-1} q.
    \end{align}
    Since $p^2 - A q^2 = 1$ and $p \equiv 1 \pmod{2}$, we must have that $Aq \equiv 0 \pmod{2}$. Thus, $p_n \equiv p_{n-1} \equiv 1 \pmod{2}$, so we can conclude that $p_n \equiv 1 \pmod{2}$ for all $n \geq 0$. A similar argument shows that $p_n$ is odd for all negative $n$ as well. Inserting this result into Equation \eqref{eq:cn-an} shows that $c_n - a_n \equiv c - a \pmod{2}$ for all $n$. In other words, the parity of all solutions $(a_n, b_n, c_n)$ generated from $(a,b,c)$ is the same parity as the base solution.
\end{proof}

We now prove Theorem \ref{T:ExistenceOfSolutions} by applying the previous proposition. We choose our base solution $(a,b,c)$ to be a parameterized solution. The proposition allows $a_n$, $b_n$, and $c_n$ to be nonpositive, so to guarantee that the generated solutions $(a_n, b_n, c_n)$ are solutions to the SoCS Problem, we must verify that $0 < a_n + 1 < b_n < c_n$.

\begin{proof}[Proof of Theorem \ref{T:ExistenceOfSolutions}]
    We specialize Proposition \ref{prop:Pell-generator} to the parameterized solutions. For any odd $N \geq 3$, there is some $k \in \N$ such that $N = 2k + 1$. Then \[
        (a,b,c) = (2k^2 + k - 1, 2k^2 + 2k, 2k^2 + 3k)
    \]
    is a parameterized solution that satisfies $c - a = 2k + 1 = N$.

    For convenience, we explicitly state the fundamental solution $(p,q)$ to the Pell equation $p^2 - A q^2 = 1$.
    We can find the fundamental solution by analyzing the continued fraction of $\sqrt{A}$ (see Section~4.8 of \cite{continued-fractions} for a proof of this method). Here we have \[
        A = k^2 (36 k^4 + 72 k^3 + 36 k^2 - 3).
    \]
    The continued fraction for $\sqrt{A}$ is given by \[
        \sqrt{A} = \lrb{6k^3 + 6k^2 - 1; \obar{1, 4k + 2, 1, 12k^3 + 12k^2 - 2}},
    \]
    where $[a_0; a_1, a_2, \dots]$ denotes the continued fraction \[
        a_0 + \dfrac{1}{a_1 + \dfrac{1}{a_2 + \dots}}
    \]
    and the bar indicates that the values are repeated infinitely.
    Since the repetend has even length, the fundamental solution is given by the fourth convergent: \[
        \frac{p}{q} = \frac{24k^4 + 48k^3 + 24k^2 - 1}{4k + 4}.
    \]

    Applying Proposition \ref{prop:Pell-generator}, we can conclude that there are infinitely many integer solutions to \eqref{E:PyramidalEquality} that lie in the plane \begin{equation}\label{eq:k-plane}
        k x - (1 + 2k) y + (1 + k) z = 0.
    \end{equation}
    We note that we obtain a different plane for each $k \in \N$ and therefore for each odd $N \geq 3$. The solutions can be written explicitly as \begin{align*}
        a_n &= -\frac{1}{2} + \frac{1}{2} (4k^2 + 2k - 1) p_n + \frac{1}{2} k(24k^4 + 36k^3 + 6k^2 - 6k - 1) q_n, \\
        b_n &= -\frac{1}{2} + \frac{1}{2} (4k^2 + 4k + 1) p_n + \frac{1}{2} k(24k^4 + 48k^3 + 30k^2 + 6k - 1) q_n, \\
        c_n &= -\frac{1}{2} + \frac{1}{2} (4k^2 + 6k + 1) p_n + \frac{1}{2} k(24k^4 + 60k^3 + 42k^2 + 6k - 1) q_n,
    \end{align*}
    where $p_n, q_n$ are obtained by taking powers of the fundamental solution, as in Proposition \ref{prop:Pell-generator}.
    For $k \geq 1$, $a_n, b_n, c_n > 0$ whenever $p_n, q_n \geq 1$. We can easily obtain a sequence of $p_n, q_n$ satisfying this bound by only taking positive powers of $p + q \sqrt{A}$. Thus, $a_n, b_n, c_n > 0$ for all $n \geq 1$. We must also show that $a_n + 1 < b_n < c_n$. For $p_n, q_n, k \geq 1$, we have \begin{align}
        c_n - b_n &= k(p_n + 6k^2 (k + 1) q_n) > 0, \\
        b_n - (a_n + 1) &= (k+1) p_n + 6k^2 (k+1)^2 q_n - 1 > 0.
    \end{align}
    Thus, $(a_n, b_n, c_n)$ is a valid solution to the SoCS problem for all $n \geq 1$. Lastly, we have $c_n - a_n \equiv c - a \equiv N \equiv 1 \pmod{2}$, so $c_n - a_n$ is odd for all $n$.
\end{proof}

Two distinct odd $N, N' \geq 3$ yield distinct planes as in \eqref{eq:k-plane}, and these planes intersect only at the line $x = y = z$. By considering all odd $N \geq 3$, we obtain the following corollary.

\begin{corollary}
    There does not exist any finite collection of planes in $\R^3$ which contains all solutions to the Sum of Consecutive Squares Problem.
\end{corollary}

We now prove Theorem \ref{T:EvenSolutions}. This theorem is weaker than Theorem \ref{T:ExistenceOfSolutions} because we do not have an infinite collection of starting solutions to use, as we had with the parameterized solutions. Instead, we use a single solution to obtain the result.

\begin{proof}[Proof of Theorem~\ref{T:EvenSolutions}]
    We apply Proposition \ref{prop:Pell-generator} to an even solution. Let $(a,b,c) = (66, 159, 198)$. We note that $c - a = 132$ is even and $A = \numprint{681522798996}$. We again explicitly state $p$ and $q$ for ease of computations. The periodic part of the continued fraction for $\sqrt{A}$ has period length $212$. We compute that the $212$th convergent $\frac{p}{q}$ is given by
    \begin{align*}
        p &= \numprint{656255818034383997445391312835392606452606438915940344809477} \\
        &\phantom{{}={}} \numprint{432517816415686303848938358579958720601124769027506828449}, \\
        q &= \numprint{794937477236090382376510634959368770095239574754242412036} \\
        &\phantom{{}={}} \numprint{127917605604471287041191179858533774984393335615406940}.
    \end{align*}
    
    Applying Proposition \ref{prop:Pell-generator}, we see that the solution $(66, 159, 198)$ generates infinitely many integer solutions $(a_n, b_n, c_n)$. From the equations for $a_n$, $b_n$, and $c_n$ given in the proof of the proposition, we have \begin{align*}
        a_n &= -\frac{1}{2} + \frac{133}{2} p_n + \numprint{54898155}\, q_n, \\
        b_n &= -\frac{1}{2} + \frac{319}{2} p_n + \numprint{131674491}\, q_n, \\
        c_n &= -\frac{1}{2} + \frac{397}{2} p_n + \numprint{163871019}\, q_n.
    \end{align*}
    As in the proof of Theorem \ref{T:ExistenceOfSolutions}, $p_n, q_n \geq 1$ for $n \geq 1$, so $a_n, b_n, c_n > 0$ for these $n$. We also have \begin{align}
        c_n - b_n &= 39 p_n + \numprint{32196528} q_n > 0, \\
        b_n - (a_n + 1) &= 93 p_n + \numprint{76776336} q_n - 1 > 0,
    \end{align}
    for $p_n, q_n \geq 1$, so $(a_n, b_n, c_n)$ is a valid solution to the SoCS Problem for all $n \geq 1$.
    These solutions satisfy $c_n - a_n \equiv c - a \equiv 0 \pmod{2}$, so $c_n - a_n$ is even for each~$n$.
\end{proof}

\begin{remark}
    In the proof of Theorem \ref{T:EvenSolutions}, we generated infinitely many even solutions by starting with the solution $(66, 159, 198)$. This choice is not unique, and any of the even solutions given in Table \ref{tab:many-non-param-solutions} may be used to prove the theorem. Some of the solutions, however, have $A$ values with a continued fraction with a much longer period, making explicit calculations of the generated solutions $(a_n, b_n, c_n)$ more difficult.
\end{remark}

\section{Arithmetic Polygons}\label{S:ArithmeticPolygons}

We now discuss the relation between solutions to the Sum of Consecutive Squares Problem and arithmetic polygons. We begin by proving Theorem~\ref{T:SolutionsPolygonsCorrespondence}. We show how any arithmetic polygon corresponds to a solution to the SoCS Problem. We then provide a process which allows one to construct arithmetic polygons from a solution. Since every side in a given polygon has a unique length, we will refer to the sides by their length (e.g., ``side $a+1$'' refers to the side of length $a+1$).

\begin{proof}[Proof of Theorem~\ref{T:SolutionsPolygonsCorrespondence}]
    Suppose we have an arithmetic polygon with side lengths running from $a+1$ through $c$. Let $P_j$ be the vertex between sides $a+j$ and $a+j+1$. Side $a+2$ must meet side $a+1$ at a right angle, otherwise the perpendicular condition for arithmetic polygons would imply that the line connecting $P_2$ to $O$ would create a right triangle with sides $a+1$ and $a+2$, with side $a+1$ as the hypotenuse, which is impossible since the hypotenuse must always be the longest side. 
    
    Next, any side $a+j$ (for $3\leq j\leq c-a-1$) must be perpendicular to either the diagonal passing through $O$ and $P_{j-1}$ or the diagonal passing through $O$ and $P_j$. Let $j'$ be the smallest $j$ such that the diagonal perpendicular to side $a+j'$ passes through $P_{j'}$. Through repeated applications of the Pythagorean Theorem, we see that the diagonal from $O$ to $P_{j'-1}$ has length
    \begin{equation}\label{E:PolyToSolLength1}
        \sqrt{\sum_{i=1}^{j'-1} (a+i)^2}.
    \end{equation}

    Furthermore, the diagonal perpendicular to side $a+j'+1$ must pass through $P_{j'+1}$. If this diagonal passed through $P_{j'}$, then sides $a+j'$ and $a+j'+1$ would be collinear, causing $P_{j'}$ to be a degenerate vertex, which violates Property (\ref{I:ArithmeticPropertyDegenerate}) of Definition~\ref{D:ArithmeticPolygons}. Continuing inductively, it is true for all $j' \leq j\leq c-a-1$ that the diagonal perpendicular to side $a+j$ passes through $P_j$. Once again, repeated applications of the Pythagorean Theorem imply that the diagonal from $O$ to $P_{j'-1}$ has length
    \begin{equation}\label{E:PolyToSolLength2}
        \sqrt{\sum_{i=j'}^{c-a} (a+i)^2}.
    \end{equation}

    Equating \eqref{E:PolyToSolLength1} and \eqref{E:PolyToSolLength2}, squaring both sides, and writing $j'=b-a+1$, we find
    \begin{equation}
        (a+1)^2 + \cdots + b^2 = (b+1)^2 + \cdots + c^2.
    \end{equation}
    Therefore, given any arithmetic polygon, we have a solution to the SoCS Problem.

    We now show the opposite correspondence.
    Arithmetic polygons can be constructed from solutions $(a,b,c)$ to the SoCS Problem using the following general process. (See Figure~\ref{F:Construction} for a step-by-step visualization of this process for constructing a polygon from the $(9, 12, 14)$ solution.)

    \begin{figure}[t]
        \centering
        \begin{subfigure}{0.29\linewidth}
            \centering
            \includegraphics[width=\textwidth]{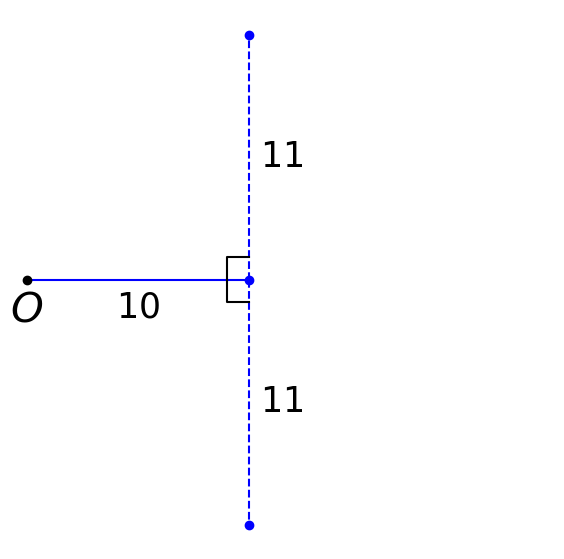}
            \caption{Side 10 sprouts out from $O$ to the east. We have two choices for side 11.}\label{F:Construction1}
        \end{subfigure}
        \hfill
        \begin{subfigure}{0.29\linewidth}
            \centering
            \includegraphics[width=\textwidth]{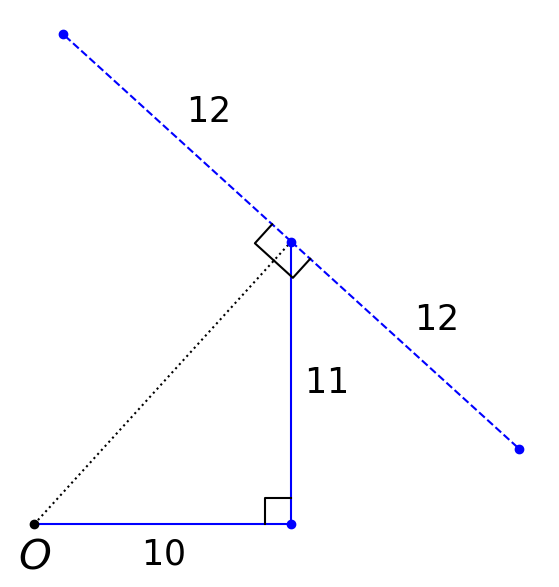}
            \caption{After making a choice for side 11, we have two possible choices for side 12.}
            \label{F:Construction2}
        \end{subfigure}
        \hfill
        \begin{subfigure}{0.29\linewidth}
            \centering
            \includegraphics[width=\textwidth]{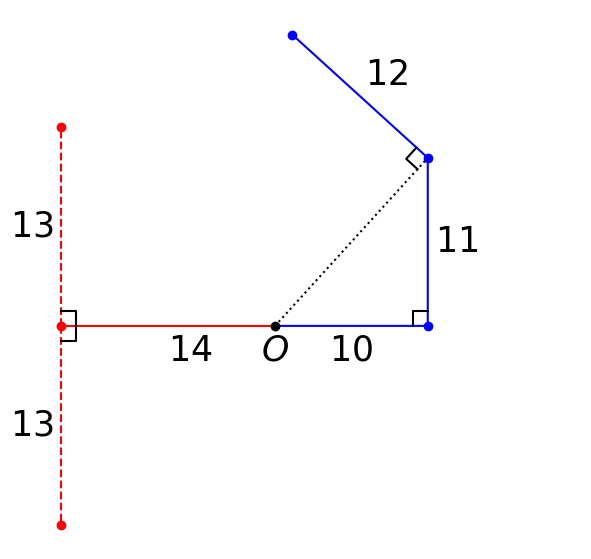}
            \caption{Side 14 sprouts out to the west of $O$. We have two choices for side 13.}\label{F:Construction3}
        \end{subfigure}
    
        \bigskip
        \begin{subfigure}{0.45\linewidth}
            \centering
            \includegraphics[width=0.67\textwidth]{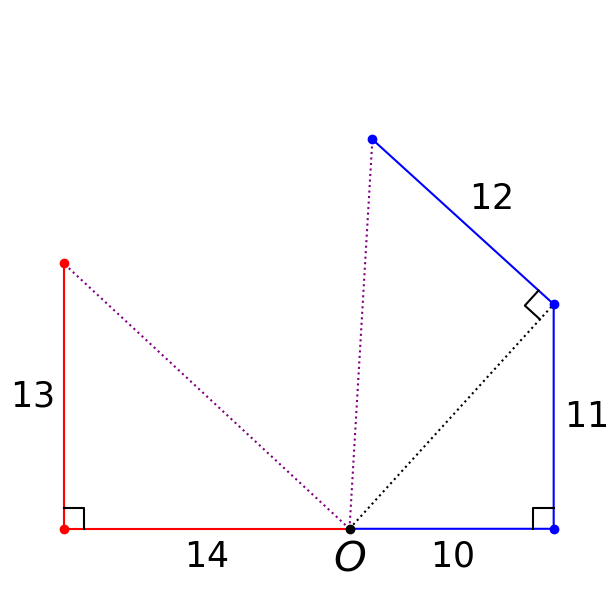}
            \caption{Both arms of the polygon completed.}\label{F:Construction4}
        \end{subfigure}
        \hfill
        \begin{subfigure}{0.45\linewidth}
            \centering
            \includegraphics[width=0.67\textwidth]{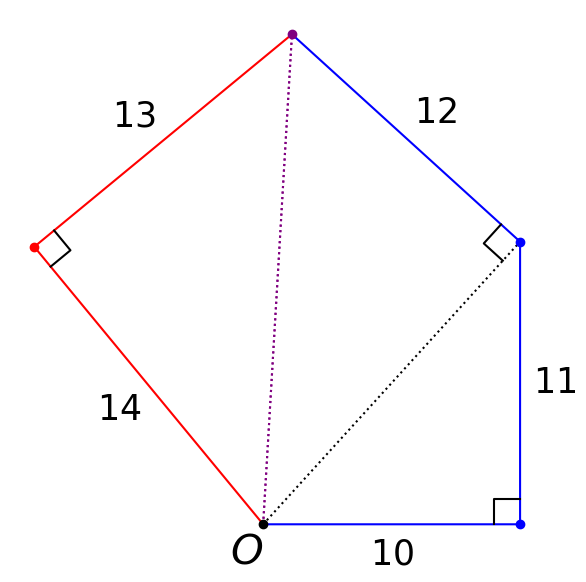}
            \caption{The second arm is rotated to meet the first arm and complete the polygon.}\label{F:Construction5}
        \end{subfigure}
        \caption{Constructing an arithmetic polygon from the $(9, 12, 14)$ solution.}\label{F:Construction}
    \end{figure}

    First, one starts off with the single vertex $O$. From here, side $a+1$ sprouts out to the east, and we let $P_1$ denote the newly created vertex at the end of this side. This side trivially satisfies Property (\ref{I:ArithmeticPropertyIntersection}) of Definition~\ref{D:ArithmeticPolygons} since $O$ is one of the vertices bounding this line. We now carry on the following inductive step to finish constructing the first ``arm'' of the polygon. 
    Suppose that sides $a+1$ through $a+j$ (with $j<b-a$) have been created, along with the respective vertices $P_1$ through $P_j$ (none of which are degenerate). Furthermore, suppose that the angle $\angle O P_{j-1} P_j$ is $\pi/2$ so that side $a+j$ satisfies Property~(\ref{I:ArithmeticPropertyIntersection}) of Definition~\ref{D:ArithmeticPolygons}, and suppose that the length of the diagonal from $O$ to $P_{j}$ is
    \begin{equation}
        \sqrt{(a+1)^2+\cdots+(a+j)^2}.
    \end{equation}
    
    We now draw the next side, having length $a+j+1$, to be perpendicular to this diagonal. There are two choices for this next side, as seen in Figure~\ref{F:Construction}, because we can go in either direction along the line passing through $P_{j}$ that is perpendicular to $\overline{O P_{j}}$. The side $a+j+1$ satisfies Property (\ref{I:ArithmeticPropertyIntersection}) of Definition~\ref{D:ArithmeticPolygons} by construction. The newly created vertex at the end of this side length is called $P_{j+1}$, and we note that the diagonal from $P_{j+1}$ to $O$ has length 
    \begin{equation}
        \sqrt{(a+1)^2+\cdots+(a+j)^2+(a+j+1)^2}.
    \end{equation}
    We also want to show that the vertex $P_j$ is not degenerate. If $j=1$, then the diagonal from $O$ to $P_1$ coincides exactly with side $a+1$, so by construction the angle $\angle OP_1P_2$ is exactly $\pi/2$. For $j>1$, we first note that by our induction hypothesis, the angle $\angle OP_{j-1}P_j$ is $\pi/2$. Furthermore, by construction the angle $\angle OP_jP_{j+1}$ is also $\pi/2$. If the angle $ \angle P_{j-1} P_j P_{j+1}$ were either $\pi$ or $0$, this would imply that the angle $\angle O P_j P_{j-1}$ would be $\pi/2$, and so the triangle $OP_{j-1}P_j$ would have two right angles. Since this is impossible, it follows that the vertex $P_j$ cannot be degenerate.
    One now continues this process until we reach vertex $P_{b-a}$.
    
    We now return to $O$, and draw a line sprouting out to the west with length $c$, and call the new vertex at the end of this side $Q_1$. We then repeat the exact same process as above, only reducing each side length by 1 rather than increasing, until we draw the side with length $b+1$, whose vertex we call $Q_{c-b}$. Since the length of the diagonal from $Q_{c-b}$ to $O$ is 
    \begin{equation}
        \sqrt{(b+1)^2+\cdots+c^2},
    \end{equation}
    and the length of the diagonal from $P_{b-a}$ to $O$ is 
    \begin{equation}
        \sqrt{(a+1)^2+\cdots+b^2},
    \end{equation}
    and $(a,b,c)$ is a solution to the SoCS Problem, the two diagonals have equal length. Therefore, we can rigidly rotate one of the arms sprouting from $O$ until $P_{b-a}$ and $Q_{c-b}$ coalesce into a single point, thus completing the polygon. 
    
    We now need to show that the vertices $O$ and $P_{b-a}$ ($=Q_{c-b}$) are not degenerate. We first show that the vertex $P_{b-a}$ is not degenerate. Consider the case where $c-b=1$. By Lemma~\ref{L:SolutionsBalanced}, $c-b=1$ implies that $c-a\leq 4$. The only solution satisfying this bound is $(2,4,5)$, which we already know has a corresponding arithmetic polygon, namely the $3$-$4$-$5$ triangle. Now, suppose that $c-b>1$. Then there is a vertex $Q_{c-b-1}$ such that the angle $\angle OQ_{c-b-1}P_{b-a}$ is $\pi/2$. Furthermore, the angle $\angle O P_{b-a-1}P_{b-a}$ is also $\pi/2$. If the angle at $P_{b-a}$ were $\pi$, then $OP_{b-a-1}Q_{c-b-1}$ would form a triangle with two right angles, which is impossible. If the angle at $P_{b-a}$ were $0$, then sides $b$ and $b+1$ lie on top of one another, so they are both line segments of the same line. The vertices $P_{b-a-1}$ and $Q_{c-b-1}$ both lie on this line, but they are not at the same point since they are at different distances from $P_{b-a}$. Therefore the angles $\angle O P_{b-a-1} P_{b-a}$ and $\angle O Q_{c-b-1} P_{b-a}$ must be different, however this is a contradiction since we know them both to be $\pi/2$. Therefore the angle at $P_{b-a}$ cannot be degenerate.

    Finally, we address the issue of degeneracy at $O$. If one follows the construction of the polygon up until this point, we begin with the angle at $O$ being $\pi$. Then when the two arms are constructed, we rotate one of the arms so that the vertices $P_{b-a}$ and $Q_{c-b}$ coalesce into a single point, thus changing the angle at $O$. If these two vertices happen to coincide without any rotation needed, then we take advantage of the freedom of choice in this process. For every side that is constructed, there is always a choice of two possible directions we can go in. Thus, we simply choose the ``other choice'' for side $b$ so that vertex $P_{b-a}$ will now lie at a different point from $Q_{c-b}$. Now some rotation is required for these two vertices to coalesce, changing the angle at $O$ away from $\pi$.
\end{proof}

\begin{remark}
    When constructing each side of the polygon (other than the initial sides of length $a$ and $c$), we have a choice whether to turn clockwise or counter-clockwise. Therefore, there are at most $2^{c-a-2}$ possible arithmetic polygons which can be constructed from any solution $(a,b,c)$. 
    Combining this upper bound with Theorem~\ref{T:BoundOnSolutionsForFixedN}, we see that for any integer $N \geq 3$, there are at most $2^{N-1}(N-1)$ possible arithmetic polygons with $N$ sides.
\end{remark}

We would like to construct polygons with the following two additional properties:

\begin{enumerate}
    \item Such polygons are convex.
    \item Such polygons are not self-intersecting (i.e., the polygon partitions the plane into a single interior and single exterior).
\end{enumerate}

We will see that the first property cannot be achieved for almost all solutions $(a,b,c)$. On the other hand, for any solution $(a,b,c)$ we can always construct an arithmetic polygon that is not self-intersecting using the following ``chainsaw process.'' See Figure~\ref{F:Chainsaws} for two examples of arithmetic polygons constructed using this process.

\begin{theorem}[Chainsaw process]
    Given any solution $(a,b,c)$ to the SoCS Problem, one can construct an arithmetic polygon with side lengths $a+1$ through $c$ which is not self-intersecting.
\end{theorem}

\begin{proof}
    We place the vertex $O$ at the origin, and place the farthest vertex from $O$, which we call $P$, at the point $(D, 0)$, where
    \begin{align*}
        D &= \sqrt{(a+1)^2 + \cdots + b^2} \\
        &= \sqrt{(b+1)^2 + \cdots + c^2}.
    \end{align*}
    
    We first give a brief overview of the construction process.
    We start by constructing the first ``arm'' of the polygon, consisting of sides $a+1$ through $b$. For each $j$ ranging from 1 to $b-a$, we draw a circle $\cC_j$ with its center at $O$ and radius equal to
    \begin{equation}
        \sqrt{\sum_{i=1}^{j}(a+i)^2}.
    \end{equation}
    Starting at $P$, each successive vertex of the polygon as one travels towards $O$ will be placed on successive circles moving inwards (so $P$ lies on the circle $\cC_{b-a}$, then traveling along side $b$ we arrive at the next vertex which lies on the circle $\cC_{b-a-1}$, then traveling along side $b-1$ we arrive at the next vertex which lies on the circle $\cC_{b-a-2}$, and so on and so forth until the vertex lying on the circle $\cC_1$ connects to $O$ via side $a+1$). Since the circles partition the plane into disjoint subsets, we are guaranteed that none of these sides will intersect with each other (other than at the vertices which lie on the circles) since they all lie in separate subsets of the plane. Finally, in order to ensure that these sides do not intersect with any side from the other arm of the polygon, we restrict every vertex in this arm to lie strictly below the $x$-axis. We will then construct the second arm in a similar manner, but with the restriction that all vertices lie above the $x$-axis. See Figure~\ref{F:Contructing_Chainsaw} for a visualization of this process.

    \begin{figure}[t]
        \centering
        \begin{subfigure}{0.45\linewidth}
            \centering
            \includegraphics[width=\textwidth]{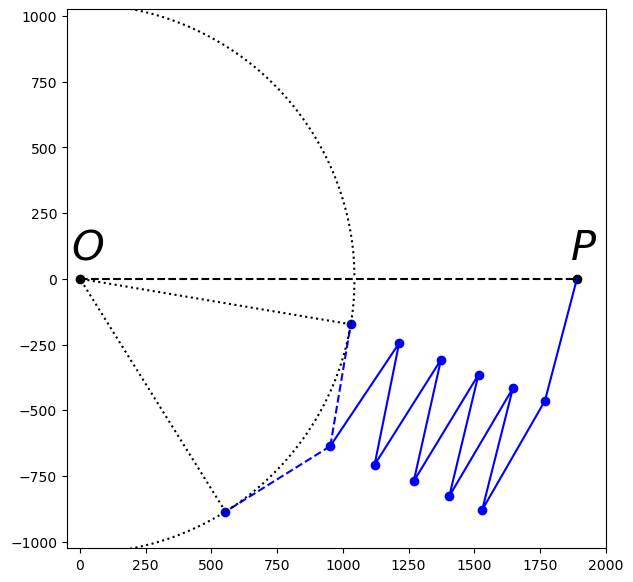}
            \caption{Constructing the lower arm of the $(464,480,495)$ chainsaw polygon. Here there are two possible locations where the next vertex could be placed.}\label{F:Chainsaw_Construction_1}
        \end{subfigure}
        \hfill
        \begin{subfigure}{0.45\linewidth}
            \centering
            \includegraphics[width=\textwidth]{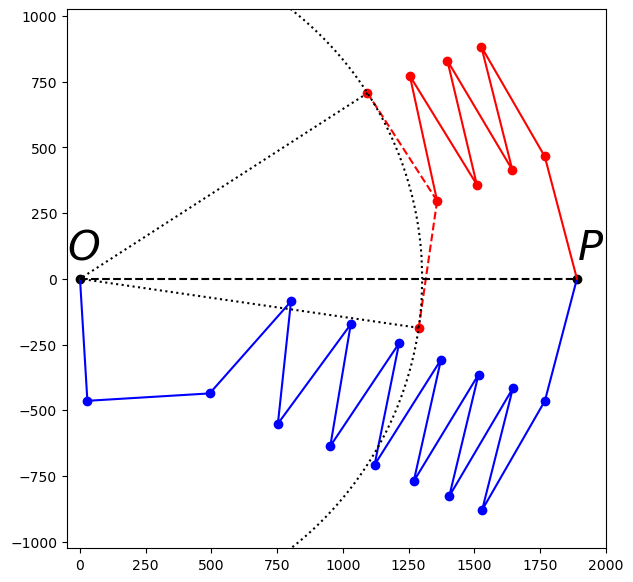}
            \caption{Constructing the upper arm of the $(464,480,495)$ chainsaw polygon. Here there is only one location that the vertex could be placed to stay above the $x$-axis.}\label{F:Chainsaw_Construction_2}
        \end{subfigure}
        \caption{Constructing an arithmetic polygon from the $(464, 480, 495)$ solution using the chainsaw process.}\label{F:Contructing_Chainsaw}
    \end{figure}

    We now prove that this process will always give an arithmetic polygon which is not self-intersecting. For the purpose of notation, let the previously undefined ``circle'' $\cC_0$ refer to the origin $O$. Suppose we have drawn the polygon up to the vertex $Q_{a+j}$, which lies below the $x$-axis and on the circle $\cC_j$. We consider the two points $Q_{a+j-1}^{(1)}$ and $Q_{a+j-1}^{(2)}$ which lie on $\cC_{j-1}$ and create a line segment with $Q_{a+j}$ which is tangent to $\cC_{j-1}$. Each of these line segments satisfy Property (\ref{I:ArithmeticPropertyIntersection}) of Definition~\ref{D:ArithmeticPolygons}. Furthermore, since (for $i\in\{1,2\}$) $Q_{a+j}Q_{a+j-1}^{(i)}O$ forms a right triangle, the line segment from $Q_{a+j}$ to $Q_{a+j-1}^{(i)}$ has length $a+j$. 

    The final thing to check is that at least one of these two points $Q_{a+j-1}^{(i)}$ lies strictly below the $x$-axis. Suppose both points lie on or above the $x$-axis. Since $Q_{a+j}$ lies below the $x$-axis, the quadrilateral $Q_{a+j}Q_{a+j-1}^{(1)}OQ_{a+j-1}^{(2)}$ will have an interior angle at $O$ which is at least $\pi$. We have a contradiction since the angles at $Q_{a+j-1}^{(i)}$ are both exactly $\pi/2$. Therefore, at least one of these points lies below the $x$-axis. If there is only one such point, then we let the next vertex in the polygon, which we denote $Q_{a+j-1}$, be that point. If both points lie below the $x$-axis, then we can choose which point we want to be our next vertex. The above argument assumes that $Q_{a+j}$ lies strictly below the $x$-axis. While $Q_{b-a}=P$ lies on the $x$-axis, it is clear that one of the tangent lines from $P$ to $\cC_{b-a-1}$ lies below the $x$-axis (while the other lies symmetrically above the $x$-axis). 

    The second arm of the polygon (consisting of sides with lengths ranging from $b+1$ through to $c$) is constructed in a similar manner, except this time with the restriction that each side and vertex lies above the $x$-axis to avoid any intersections with sides from the first arm. We construct circles centered at $O$ with radii equaling
    \begin{equation}
        \sqrt{\sum_{i=0}^{j} (c-i)^2},
    \end{equation}
    where $j$ ranges from $0$ to $c-b-1$, then proceed as we did with the first arm, starting with side $b+1$ and finishing with side $c$. See Figure \ref{F:Chainsaws} for two completed chainsaw polygons.

    \begin{figure}[t]
        \begin{subfigure}{0.45\linewidth}
            \centering
            \includegraphics[width=\textwidth]{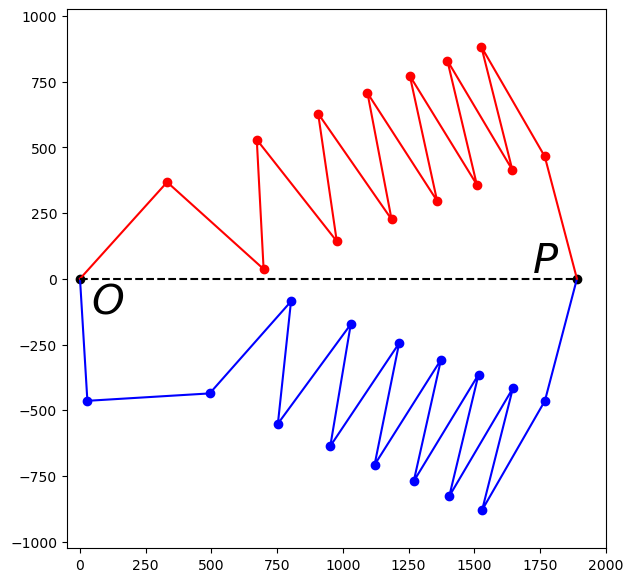}
            \caption{The $(464, 480, 495)$ chainsaw polygon.}\label{F:Chainsaw_464_480_495}
        \end{subfigure}
        \hfill
        \begin{subfigure}{0.45\linewidth}
            \centering
            \includegraphics[width=\textwidth]{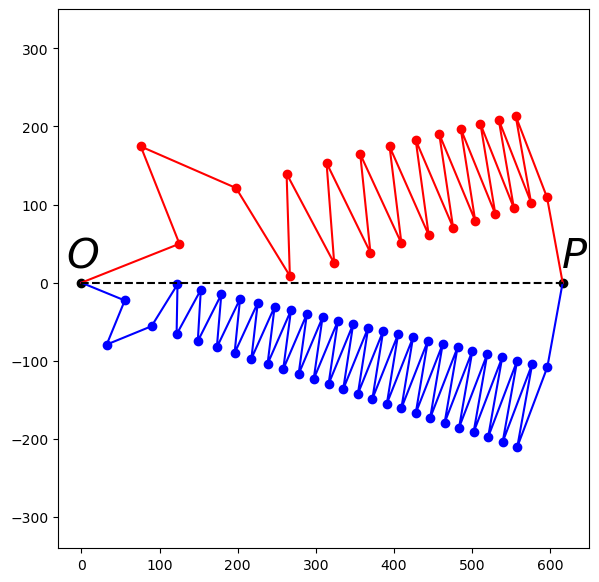}
            \caption{The $(59, 110, 135)$ chainsaw polygon.}\label{F:Chainsaw_59_110_135}
        \end{subfigure}
        \caption{Two examples of polygons constructed from the chainsaw process.}\label{F:Chainsaws}
    \end{figure}

    We now verify that none of the angles are degenerate. The angle at $P$ cannot be $\pi$ since the vertices on either side of $P$ both lie strictly inside the circle centered at the origin with radius $D$, while $P$ lies on this circle. Also, the angle at $P$ cannot be $0$ since the vertices on either side of $P$ lie on either side of the $x$-axis, while $P$ is exactly on the $x$-axis. For any vertex $Q_{a+j}$ created during the construction of the first arm, we  note that the line from the previous vertex $Q_{a+j+1}$ to $Q_{a+j}$ is tangent to the circle $\cC_j$, upon which the vertex $Q_{a+j}$ lies. If the angle at $Q_{a+j}$ were either $\pi$ or $0$, then the point $Q_{a+j-1}$ would lie outside $\cC_j$, which is a contradiction. Thus, the angle at $Q_{a+j}$ cannot be degenerate. Similarly, the vertices in the second arm also cannot be degenerate.
    
    Finally, we address the issue of degeneracy at $O$. One can computationally verify that, for all solutions $(a,b,c)$ with $c-a\leq14$, the polygon constructed using this process does not have a degenerate vertex at $O$. We now prove that the vertex $O$ is not degenerate whenever $c - a \geq 15$.
    
    Firstly, by Lemma~\ref{L:SolutionsBalanced}, if $c-b\leq3$, then
    \begin{equation}
        b-a \leq (c-b)(1+2^{1/3}+2^{2/3}) \leq 3(1+2^{1/3}+2^{2/3}) \approx 11.542.
    \end{equation}
    Therefore, $c-a\leq14$. Since $c-a$ and $c-b$ are both integers, the contrapositive of this is that if $c-a\geq 15$, then $c-b\geq 4$.

    We assume that the first arm (sides $a+1$ through $b$) of the polygon has been constructed and we focus on the construction of the second arm. To show that $O$ is not degenerate, we consider the size of the ``angular step'' that one takes when moving from one vertex to another while constructing the upper arm of the polygon. For $0 \leq j \leq c-b-1$, let $\cC'_j$ denote the circle centered at $O$ of radius \begin{equation}
        \sqrt{\sum_{i=0}^j (c-i)^2}.
    \end{equation}
    Let $Q'_{c-j}$ be the vertex of the chainsaw polygon that lies on this circle.
    The next vertex we create will be $Q'_{c-j+1}$ and will lie on the circle $\cC'_{j-1}$. We call $\angle Q'_{c-j} O Q'_{c-j+1}$ the ``angular step'' from $Q'_{c-j}$ to $Q'_{c-j+1}$. Recall that when constructing $Q'_{c-j+1}$, there are two possible points we could choose (although it is possible that one of these points could lie below the $x$-axis, in which case we are forced to choose the other). Each of these two possible points will have the same angular step from $Q'_{c-j}$, one in the clockwise direction and the other in the counter-clockwise direction. We make the following two claims:
    \begin{enumerate}
        \item For any integer $j\geq1$, the angular step from $Q'_{c-j}$ to $Q'_{c-j+1}$ is less than $\pi/4$. \label{I:DegeneracyClaimUpper}
        \item The angular steps from $Q'_{c-3}$ to $Q'_{c-2}$ and from $Q'_{c-2}$ to $Q'_{c-1}$ are both greater than $\pi/8$. \label{I:DegeneracyClaimLower}
    \end{enumerate}
    Consider the sector
    \begin{equation}
        \cS := \{(x,y) \in \R^2 : \abs{x} < y\}.
    \end{equation}
    In other words, $\cS$ contains all points whose angle against the positive $x$-axis is between $\pi/4$ and $3\pi/4$.
    If we can guarantee that $Q'_{c-1}$ lies within $\cS$, then making an angular step less than $\pi/4$ in either direction will still keep the next point above the $x$-axis. Therefore, if Claim~(\ref{I:DegeneracyClaimUpper}) is true, then we can guarantee that there are two possible valid choices for vertex $Q'_{c}$. Since it is impossible for both of those points to create a degenerate vertex at $O$, at least one of those choices will allow us to avoid degeneracy at $O$. 
    
    In order to ensure that the vertex $Q'_{c-1}$ is in $\cS$, we use both claims. Firstly, one constructs the second arm of the polygon as usual until reaching vertex $Q'_{c-3}$ (which must exist since $c-b\geq4$). If $Q'_{c-3}$ is in $\cS$, then, since the angular step to $Q'_{c-2}$ is less than $\pi/4$, one can guarantee by choosing an appropriate direction (either clockwise or counter-clockwise) that $Q'_{c-2}$ is also in $\cS$, and then similarly one can guarantee that $Q'_{c-1}$ also lies in $\cS$. If the point $Q'_{c-3}$ lies outside $\cS$ then we choose the next point such that we move towards $\cS$. If $Q'_{c-2}$ is now inside $\cS$, we can ensure that $Q'_{c-1}$ is also inside $\cS$ as described above. If $Q'_{c-2}$ is not inside $\cS$, then we again choose the next point to move in the direction of $\cS$. Since the angular steps from $Q'_{c-3}$ to $Q'_{c-2}$ and from $Q'_{c-2}$ to $Q'_{c-1}$ are both greater than $\pi/8$, their combined step is greater than $\pi/4$, ensuring that these two steps are sufficient to ensure that $Q'_{c-1}$ lies in $\cS$. For a visual representation of this procedure, see Figure~\ref{F:ChainsawDegeneracies}.

    \begin{figure}[t]
    \centering
        \begin{subfigure}{0.45\linewidth}
            \centering
            \includegraphics[width=\textwidth]{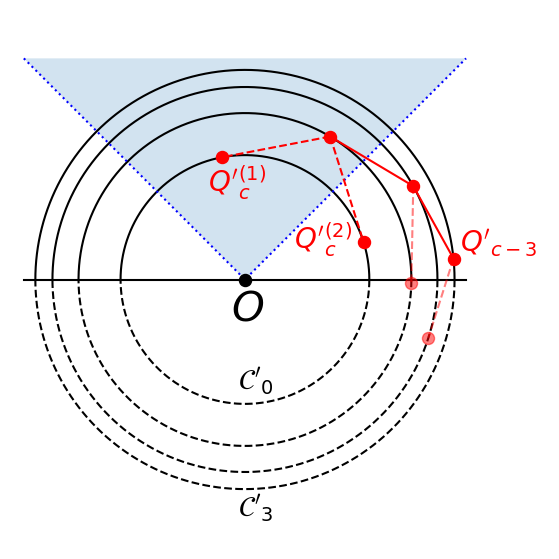}
            \caption{Here we start with $Q'_{c-3}$ outside of $\cS$, but we are able to enter the sector by taking two steps.}\label{F:ChainsawDegeneracy1}
        \end{subfigure}
        \hfill
        \begin{subfigure}{0.45\linewidth}
            \centering
            \includegraphics[width=\textwidth]{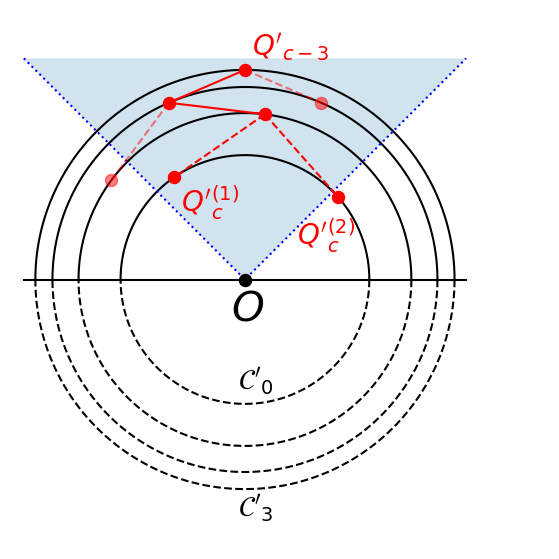}
            \caption{Here we start with $Q'_{c-3}$ inside $\cS$, and we are always able to remain in $\cS$ with each new vertex.}\label{F:ChainsawDegeneracy2}
        \end{subfigure}
        \caption{In both cases, we see that vertex $Q'_{c-1}$ lies within $\cS$, giving two possibilities for $Q'_c$ which are both above the $x$-axis.}\label{F:ChainsawDegeneracies}
    \end{figure}

    All that remains is to prove these two claims. To do this, we use the following geometric consideration. Suppose there are two circles with radii $x$ and $y$ (where $x<y$), which share a common center, which we call $O$. Let $Y$ be a point lying on the outer circle, and let $X$ be a point lying on the inner circle such that the line passing through $Y$ and $X$ is tangent to the inner circle. Let $\alpha$ be the angle $\angle YOX$. Then
    \begin{equation}
        \alpha = \arccos\left(\frac{x}{y}\right).
    \end{equation}
    Set $\eta = y/x > 1$. Then
    \begin{equation}
        \alpha = \arccos\left(\frac{1}{\eta}\right).
    \end{equation}
    We make the following three observations: (i) if $\eta=\sqrt{2}$, then $\alpha=\pi/4$; (ii) if $\eta=2/\sqrt{2 + \sqrt{2}}$, then $\alpha=\pi/8$; and (iii) since $\arccos$ is a decreasing function, $\alpha$ is increasing with respect to $\eta$.

    We begin by proving Claim~(\ref{I:DegeneracyClaimUpper}). Consider the angular step from $Q'_{c-j}$ to $Q'_{c-j+1}$, which we will denote by $\alpha$. We wish to show that
    \begin{equation}
        \sqrt{\frac{\sum_{i=0}^j (c-i)^2}{\sum_{i=0}^{j-1} (c-i)^2}} < \sqrt{2},
    \end{equation}
    which would imply that $\alpha < \pi/4$.
    We obtain this bound immediately by observing that for any $j \geq 1$, \begin{equation}
        (c-j)^2 < \sum_{i=0}^{j-1} (c-i)^2.
    \end{equation}

    We now prove Claim~(\ref{I:DegeneracyClaimLower}). For the step from $Q'_{c-3}$ to $Q'_{c-2}$, we wish to show that
    \begin{equation}
        \sqrt{\frac{c^2+(c-1)^2+(c-2)^2+(c-3)^2}{c^2+(c-1)^2+(c-2)^2}} > \frac{2}{\sqrt{2 + \sqrt{2}}},
    \end{equation}
    which would imply that $\alpha > \pi/8$.
    Firstly, since $c-b\geq4$, we know that $b-a\geq5$. Combining these bounds with $a\geq0$, we see that $c\geq9$. Secondly, the function on the left hand side of the above inequality is increasing for $c\geq 9$. Therefore,
    \begin{equation}
        \sqrt{\frac{c^2+(c-1)^2+(c-2)^2+(c-3)^2}{c^2+(c-1)^2+(c-2)^2}} \geq \sqrt{\frac{9^2+8^2+7^2+6^2}{9^2+8^2+7^2}} = \sqrt{\frac{230}{194}} > \frac{2}{\sqrt{2 + \sqrt{2}}}.
    \end{equation}

    Similarly, for the step from $Q'_{c-2}$ to $Q'_{c-1}$, one can see that
    \begin{equation}
        \sqrt{\frac{c^2+(c-1)^2+(c-2)^2}{c^2+(c-1)^2}} \geq \sqrt{\frac{194}{145}} > \frac{2}{\sqrt{2 + \sqrt{2}}},
    \end{equation}
    as desired.
\end{proof}

\begin{remark}
    In Figure \ref{F:Chainsaws}, whenever we had two valid choices for the next side in our construction, we always chose to go towards the $x$-axis. This aesthetic choice, while not necessary, gives the polygons a jagged appearance and is the inspiration for the name ``chainsaw.''
\end{remark}

We now address the issue of convexity. A polygon is convex if and only if all of the interior angles are less than $\pi$. For any polygon $\cP$, let $\mu(\cP)$ be the number of interior angles of $\cP$ which are greater than $\pi$. Therefore, any convex polygon $\cP$ will satisfy $\mu(\cP)=0$, and in general one can think of $\mu(\cP)$ as a measure of how far $\cP$ is from being convex. We see in the following theorem that not only are convex arithmetic polygons rare, but so are almost-convex arithmetic polygons.

\begin{theorem}\label{T:ConvexityBound}
    Let $\mu$ be the counting function defined above, and let $\cP$ be an arithmetic polygon with $N$ sides. For any $\nu$, $\mu(\cP) \leq \nu$ only if $N\leq8+3\pi^2(\nu+2)^2$.
\end{theorem}

\begin{proof}
    Suppose we have an arithmetic polygon with $N$ sides whose smallest side has length $a+1$. From Theorem \ref{T:SolutionsPolygonsCorrespondence}, this polygon corresponds to a solution to the SoCS problem, and we will denote this solution as $(a,b,c)$. We note that $N = c - a$. For any $1<\ell<c-a$, consider the right triangle constructed from $O$ and side $a+\ell$. The hypotenuse of this triangle will have length $u$, where
    \begin{align*}
        u^2 &\leq (a+1)^2+\cdots + (a+\ell)^2 \\
        &= a^2\ell + a\ell(\ell+1) + \frac{1}{6}\ell(\ell+1)(2\ell+1) \\
        &\leq a^2\ell + 2a\ell^2 + \ell^3 \\
        &\leq 3(a^2\ell + \ell^3).
    \end{align*}
    Hence, $u\leq\sqrt{3\ell}(a+\ell)$. Let $\alpha_{\ell}$ be the angle of this triangle at $O$. Then,
    \begin{equation*}
        \alpha_\ell \geq \sin\alpha_\ell
        = \frac{a+\ell}{u}
        \geq \frac{1}{\sqrt{3\ell}}.
    \end{equation*}

    We now turn our attention to the interior angles of the polygon. For $1\leq\ell<c-a$, let $\beta_{\ell}$ be the interior angle between sides $a+\ell$ and $a+\ell+1$. The angle at $O$ shall be denoted $\beta_0$.

    If $1<\ell<b-a$, then side $a+\ell+1$ is perpendicular to the hypotenuse of the triangle constructed from $O$ and the side $a+\ell$. Therefore, if $\beta_{\ell} < \pi$, then either $\beta_{\ell} = \pi - \alpha_{\ell}$ (if the interior of the triangle is also the interior of the polygon) or $\beta_{\ell} = \alpha_{\ell}$ (if the interior of the triangle is the exterior of the polygon). See Figure \ref{F:Convexity_Demos}. Since $\alpha_{\ell} < \pi/2$, we have $\alpha_{\ell} < \pi - \alpha_{\ell}$, so in either case,
    \begin{equation}
        \beta_{\ell} \leq \pi - \alpha_{\ell} \leq \pi - \frac{1}{\sqrt{3\ell}} \leq \pi - \frac{1}{\sqrt{3N}}.
    \end{equation}
    On the other hand, if $\beta_{\ell} > \pi$, then either $\beta_{\ell} = 2\pi - \alpha_{\ell}$ or $\beta_{\ell} = \pi + \alpha_{\ell}$. Similarly, we see that in either case,
    \begin{equation}
        \beta_{\ell} \leq 2\pi - \alpha_{\ell} \leq 2\pi - \frac{1}{\sqrt{3\ell}} \leq 2\pi - \frac{1}{\sqrt{3N}}.
    \end{equation}

    \begin{figure}[t]
        \centering
        \begin{subfigure}{0.45\linewidth}
            \includegraphics[width=\textwidth]{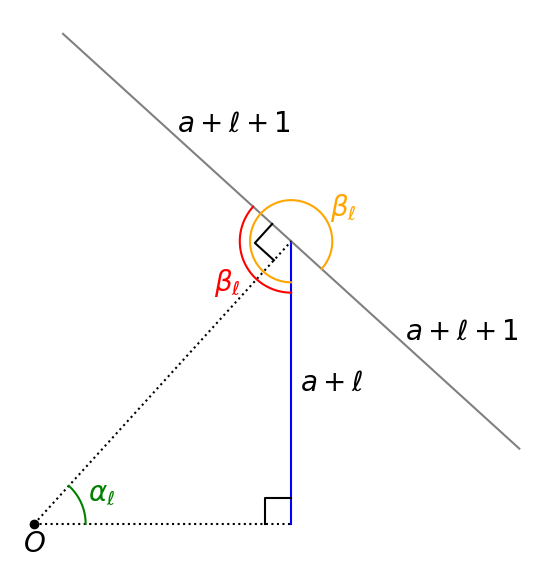}
            \caption{In this case, the interior of the polygon is also the interior of the triangle. If $\beta_{\ell} < \pi$, then $\beta_{\ell} = \pi - \alpha_{\ell}$, whereas if $\beta_{\ell} > \pi$, then $\beta_{\ell} = 2\pi - \alpha$.}\label{F:Convexity_Demo_1}
        \end{subfigure}
        \hfill
        \begin{subfigure}{0.45\linewidth}
            \includegraphics[width=\textwidth]{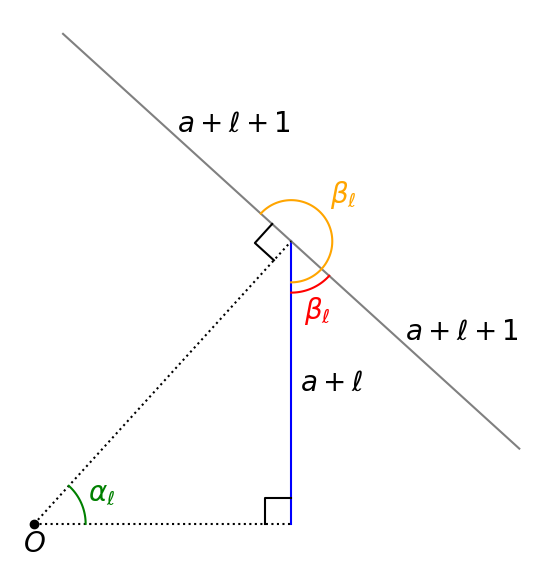}
            \caption{In this case, the interior of the polygon is the exterior of the triangle. If $\beta_{\ell} < \pi$, then $\beta_{\ell} = \alpha_{\ell}$, whereas if $\beta_{\ell} > \pi$, then $\beta_{\ell} = \pi + \alpha$.}\label{F:Convexity_Demo_2}
        \end{subfigure}
        \caption{Four possible values for $\beta_{\ell}$.}\label{F:Convexity_Demos}
    \end{figure}

    If $b-a+1<\ell<c-a$, the side of length $a+\ell-1$ is perpendicular to the hypotenuse of the triangle constructed from $O$ and the side $a+\ell$. So, similar to above, if $\beta_{\ell-1} < \pi$, then
    \begin{equation}
        \beta_{\ell-1} \leq \pi - \alpha_{\ell} \leq \pi - \frac{1}{\sqrt{3\ell}} \leq \pi - \frac{1}{\sqrt{3N}},
    \end{equation}
    whereas, if $\beta_{\ell} > \pi$, then
    \begin{equation}
        \beta_{\ell-1} \leq 2\pi - \alpha_{\ell} \leq 2\pi - \frac{1}{\sqrt{3\ell}} \leq 2\pi - \frac{1}{\sqrt{3N}}.
    \end{equation}
    So, as long as $\ell\not\in\{0, 1, b-a, c-a-1\}$, we have
    \begin{equation}
        \beta_{\ell} \leq \pi - \frac{1}{\sqrt{3N}} + \begin{cases}
            0 & \text{if } \beta_{\ell} < \pi, \\
            \pi & \text{if } \beta_{\ell} > \pi.
        \end{cases}
    \end{equation}
    Moreover, if $\ell \in \{0, 1, b-a, c-a-1\}$, then trivially
    \begin{equation}
        \beta_{\ell} \leq \pi + \begin{cases}
            0 & \text{if } \beta_{\ell} < \pi, \\
            \pi & \text{if } \beta_{\ell} > \pi.
        \end{cases}
    \end{equation}
    Therefore,
    \begin{equation}
        \pi(N-2) \leq N\lrp{\pi - \frac{1}{\sqrt{3N}}} + \frac{4}{\sqrt{3N}} + \mu(\cP)\pi,
    \end{equation}
    where the left hand side is the sum of all of the interior angles.
    Rearranging, we find
    \begin{equation}\label{eq:mu-ineq}
        \mu(\cP) \geq \frac{\sqrt{N}}{\pi\sqrt{3}}-\frac{4}{\pi\sqrt{3N}}-2.
    \end{equation}

    Suppose that $\nu$ is some integer such that 
    \begin{equation}\label{eq:n-k-ineq}
        N > 8 + 3\pi^2(\nu+2)^2.
    \end{equation}
    We must show that $\mu(\cP) > \nu$. Since $\mu(\cP) \geq 0$ always, the inequality follows trivially if $\nu < 0$. We may therefore assume that $\nu \geq 0$.
    We weaken the inequality by adding a negative number, giving us \begin{equation}
        N > 8 + 3\pi^2(\nu+2)^2 > 8 - \frac{16}{N} + 3\pi^2(\nu+2)^2.
    \end{equation}
    Rearranging, we have \begin{equation}
        N - 8 + \frac{16}{N} > 3\pi^2(\nu+2)^2.
    \end{equation}
    Since $N$ is positive, the left-hand side can be realized as a perfect square: \begin{equation}
        \lrp{\sqrt{N} - \frac{4}{\sqrt{N}}}^2 > 3\pi^2(\nu+2)^2.
    \end{equation}
    From \eqref{eq:n-k-ineq}, we have $N > 126$ since $\nu \geq 0$. The expressions within the squares on both sides of the inequality are therefore positive, so we must have \begin{equation}
        \sqrt{N} - \frac{4}{\sqrt{N}} > \pi\sqrt{3} (\nu+2).
    \end{equation}
    Solving for $\nu$, we see that \begin{equation}
        \frac{\sqrt{N}}{\pi\sqrt{3}}-\frac{4}{\pi\sqrt{3N}}-2 > \nu.
    \end{equation}
    Combining this inequality with \eqref{eq:mu-ineq} yields the desired result that $\mu(\cP) > \nu$.
\end{proof}

We have discussed how one can find arithmetic polygons from a solution $(a,b,c)$ and that one can always find an arithmetic polygon that is not self-intersecting. However, only a finite number of arithmetic polygons are convex. We now prove Theorem~\ref{T:Only2ConvexPolygons}, stating that there are only two distinct (up to rigid transformations) convex arithmetic polygons.

\begin{proof}[Proof of Theorem~\ref{T:Only2ConvexPolygons}]
    We make the following two observations:
    \begin{enumerate}
        \item An arithmetic polygon can only be convex if it has at most 126 sides.
        \item While any solution $(a,b,c)$ can produce many distinct arithmetic polygons, only at most one (up to rigid transformations) can be convex. Furthermore, we can give a description of how it is constructed.
    \end{enumerate}
    The first observation comes from applying Theorem~\ref{T:ConvexityBound} with $\nu=0$. By applying Theorem~\ref{T:BoundOnSolutionsForFixedN} and summing over $N$ from 3 to 126, we see that there are at most 
    \begin{equation}
        \sum_{N=3}^{126} 2(N-1) = 2\sum_{M=2}^{125} M = \numprint{15748}
    \end{equation}
    possible solutions which could produce a convex arithmetic polygon. The proof of Theorem~\ref{T:BoundOnSolutionsForFixedN} provides us with an algorithm for finding every solution with $N \leq 126$. Applying this algorithm, we find that there are only 67 solutions which could produce a convex arithmetic polygon. These solutions are all the parameterized solutions $(2k^2 + k - 1, 2k^2 + 2k, 2k^2 + 3k)$ for $k \leq 62$ and the following non-parameterized solutions: 
    \begin{align}
        (17, 34, 42), & && (3, 38, 48), & (11, 50, 63), & && (59, 110, 135), & (66, 159, 198).
    \end{align}
    
    To give a bound on the number of convex polygons, we require the second observation. To see that there is only one convex polygon for each solution, we consider the algorithm described for constructing arithmetic polygons given in the second half of the proof of Theorem~\ref{T:SolutionsPolygonsCorrespondence}.

    When constructing side $a+2$, we have the choice of two possible directions in which we can move. Once we make this choice, the choices for the remaining sides are determined. For example, in Figure~\ref{F:Construction2}, which shows a step in constructing a polygon from the $(9, 12, 14)$ solution, we see that there are two possibilities for side $13$. However, only one of them gives an interior angle less than $\pi$, while the other gives an interior angle greater than $\pi$. In other words, once we have a defined interior of our polygon, every new side created must always turn inwards towards the interior to create an angle smaller than $\pi$. Therefore, for a given solution, there are only two polygons that can be constructed which could possibly be convex, one for each choice of side $a+2$. However, these polygons are reflections of one another, so we are left with only one distinct polygon which could be convex.

    While it is necessary for a convex arithmetic polygon to have been constructed in this way, it is certainly not sufficient. For example, when rotating one arm to meet the other as is done in the proof of Theorem \ref{T:SolutionsPolygonsCorrespondence}, it is possible that the angle at $O$ is greater than $\pi$. Alternatively, when the two end points of each arm coalesce into a single vertex, this angle could also be greater than $\pi$. Additionally, this construction gives no guarantee that the polygon would not be self-intersecting.

    Since we have a finite list of all possible solutions which could produce a convex polygon, and for every solution we have an algorithm to construct the only polygon which could possibly be convex, one can create each of these polygons and see case-by-case whether each one is convex or not. Doing so, we conclude that there are only two convex arithmetic polygons: one coming from the solution $(2,4,5)$ and one coming from the solution $(9, 12, 14)$. The polygon corresponding to $(2,4,5)$ is the classic $3$-$4$-$5$ triangle, while the polygon corresponding to $(9, 12, 14)$ is the pentagon constructed in Figure \ref{F:Construction}.
\end{proof}

\vskip20pt\noindent {\bf Acknowledgement.} This material is based upon work supported by the National Science Foundation Graduate Research Fellowship Program under Grant No. DGE 21-46756. Any opinions, findings, and conclusions or recommendations expressed in this material are those of the authors and do not necessarily reflect the views of the National Science Foundation.

\printbibliography

\end{document}